\documentclass[11pt, a4paper]{article}
\usepackage{times}
\usepackage{hyperref}
\usepackage[british]{babel}
\usepackage{enumerate, longtable, multirow}
\usepackage{amsmath, amscd, amsfonts, amsthm, amssymb, latexsym, comment, stmaryrd, graphicx, color}
\usepackage[T1]{fontenc}
\usepackage[utf8]{inputenc}
\usepackage{fullpage}

\newtheorem{thm}{Theorem}[section]
\newtheorem{lem}[thm]{Lemma}
\newtheorem{defi}[thm]{Definition}
\newtheorem{rem}[thm]{Remark}

\newtheorem{prop}[thm]{Proposition}

\newcommand{\GL}{\mathrm{GL}}

\newcommand{\PGL}{\mathrm{PGL}}

\newcommand{\Ind}{{\rm Ind}}

\DeclareMathOperator{\Image}{im}
\DeclareMathOperator{\Gal}{Gal}

\DeclareMathOperator{\Frob}{Frob}
\DeclareMathOperator{\tr}{tr}

\newcommand{\mat}[4]{
\left(  \begin{smallmatrix} #1 & #2 \\ #3 & #4 \end{smallmatrix} \right)}

\newcommand{\cO}{\mathcal{O}}

\newcommand{\fa}{\mathfrak{a}}
\newcommand{\fb}{\mathfrak{b}}
\newcommand{\fc}{\mathfrak{c}}

\newcommand{\ff}{\mathfrak{f}}

\newcommand{\fm}{\mathfrak{m}}
\newcommand{\fn}{\mathfrak{n}}

\newcommand{\fp}{\mathfrak{p}}

\newcommand{\CC}{\mathbb{C}}

\newcommand{\QQ}{\mathbb{Q}}

\newcommand{\RR}{\mathbb{R}}

\newcommand{\ZZ}{\mathbb{Z}}

\begin{document}

\title{Köhler's Conjecture on Hecke Theta Series of Weight One}
\author{Mahima Kumar, Gabor Wiese}

\maketitle

\begin{abstract}
In this article, we prove a conjecture of Günter Köhler on the ambiguity of the quadratic field in the definition of Hecke theta series by deriving it from a similar statement on two-dimensional Galois representations induced from characters of quadratic fields.

\end{abstract}

\section{Introduction} \label{sec:intro}

By definition, a Hecke theta series is associated with a Hecke character of a quadratic field.
It is natural to ask under which conditions the quadratic field is uniquely determined by the theta series.
This question is addressed in a conjecture by Günter Köhler in his book (\cite[Conjecture~5.5]{Koehler-book}) and, in a more precise form, in the subsequent article~\cite{Koehler-article}.

In this paper we prove a slightly modified version of Köhler's conjecture (see Theorem~\ref{thm:conj}).
We derive it from a completely analogous question on $2$-dimensional representations of a group that are inductions of a character of an index-$2$ subgroup (see Proposition~\ref{prop:rep}). This group theoretic result is well known (e.g. \cite[Propositions 9, 10]{R-AlmAb}, \cite[Proposition 2.3]{SchmidtTurki}, where the representations are called triply monomial and triply imprimitive, respectively).
Our presentation is independent, very explicit with concrete matrix realisations and follows the structure of Köhler's conjecture.

We also determine all possible images of $2$-dimensional group representations that can be induced by characters of two different index-$2$ subgroups (see Proposition~\ref{prop:image}). This provides a different perspective on the list of almost abelian $2$-groups with a faithful $2$-dimensional representation from \cite[Proposition~18]{R-AlmAb}.

\subsection*{Acknowledgments}
The authors would like to thank David Rohrlich for having shared the article~\cite{R-AlmAb}.
The second-named author would like to thank Kathrin Bringmann for having brought Köhler's conjecture to his attention many years ago.

\section{Representation theoretic version of the conjecture}
This section contains the representation theoretic core of the paper. We start by introducing some notation and collecting the facts about induced representations that will be used in the sequel.

For a finite group $G$, a normal subgroup $H \subseteq G$, a character $\chi: H \to \CC^\times$ and $\sigma \in G$, we define the conjugated character $\chi^\sigma: H \to \CC^\times$ by setting $\chi^\sigma(h) = \chi(\sigma^{-1}h\sigma)$ for $h \in H$.
We now specialise to index~$2$ and write $G = H \sqcup \sigma H$. Then $\chi^\sigma$ does not depend on the choice of $\sigma \in G \setminus H$ and, for a suitable choice of basis, the induced representation $\rho = \Ind_H^G(\chi)$ can be explicitly described by
\begin{equation}\label{eq:induced}
\rho(g) = \begin{cases}
\mat {\chi(h)}00{\chi^\sigma(h)} & \textnormal{ if } g=h \in H,\\[.3cm]
\mat 0 {\chi(\sigma^2)\chi^\sigma(h)}{\chi(h)}0 & \textnormal{ if } g=\sigma h \notin H.
\end{cases}
\end{equation}
By Mackey's criterion, $\rho$ is irreducible if and only if $\chi \neq \chi^\sigma$.

In order to ease notation, we shall now without loss of generality assume that $\rho$ is faithful, allowing us to see $G$ inside $\GL_2(\CC)$ via~$\rho$.
We define the character $\epsilon_\chi:H \to \CC^\times$ as $\epsilon_\chi = \chi^\sigma / \chi = \det(\rho)|_H/\chi^2$.
Then $\rho$ is irreducible if and only if $\epsilon_\chi$ is not the trivial character.
Let $K = \ker(\epsilon_\chi) \subseteq H$ be its kernel. If $\epsilon_\chi$ is non-trivial, then $K = Z(G)$ is the centre of~$G$.

\begin{prop}\label{prop:rep}
The following statements are equivalent.
\begin{enumerate}[(A')]
\item \label{prop:rep:A} There exist a subgroup $H' \subset G$ of index~$2$ different from $H$ and a character $\chi': H' \to \CC^\times$ such that $\rho \cong \Ind_{H'}^G (\chi')$.
\item \label{prop:rep:B} $\rho$ is irreducible and for all $h\in H$ we have $\chi(h)=\chi^\sigma(h)$ or $\chi(h)=-\chi^\sigma(h)$.
\item \label{prop:rep:C} $\epsilon_\chi$ is quadratic.
\item \label{prop:rep:D} $\tr(\rho(h)) = 0$ for all $h \in H \setminus K$.
\item \label{prop:rep:E} $G/K \cong C_2 \times C_2$.
\item \label{prop:rep:F} There exist two subgroups $H' \subset G$ of index~$2$ different from $H$ and characters $\chi': H' \to \CC^\times$ such that $\rho \cong \Ind_{H'}^G (\chi')$.
\end{enumerate}
\end{prop}

\begin{proof}
\underline{(\ref{prop:rep:B}') $\Leftrightarrow$ (\ref{prop:rep:C}') $\Leftrightarrow$ (\ref{prop:rep:D}'):} Clear from the previous discussion, the definition of $\epsilon_\chi$ and the formula $\tr(\rho(h)) = \chi(h) + \chi^\sigma(h)$, following from~\eqref{eq:induced}.

\underline{(\ref{prop:rep:C}') $\Rightarrow$ (\ref{prop:rep:E}'):} From \eqref{eq:induced} it is clear that the representation $\Ind_H^G(\epsilon_\chi)$ has kernel~$K$ and image isomorphic to $C_2 \times C_2$.

\underline{(\ref{prop:rep:E}') $\Rightarrow$ (\ref{prop:rep:C}'):} As $H/K$ is isomorphic to $C_2$ and $K$ is $\ker(\epsilon_\chi)$, the character $\epsilon_\chi$ has to be quadratic.

\underline{(\ref{prop:rep:E}') $\Rightarrow$ (\ref{prop:rep:F}'):} There are two subgroups of $G/K \cong C_2 \times C_2$ of order~$2$ different from $H/K$, their preimages under the projection $G \twoheadrightarrow G/K$ giving the two claimed subgroups of~$G$ of index~$2$ different from~$H$.
Let $H'$ be one of them and let $\sigma' \in H' \setminus H$. As $\sigma'^2 \in K \subset H$, we can further choose $s \in \CC$ to be a square root of $\chi(\sigma'^2)$. This allows us to define a map $\chi': H' \to \CC^\times$ by setting $\chi'(k) = \chi(k)$ and $\chi'(\sigma' k) = s \cdot \chi(k)$. It is now easily checked that $\chi'$ is a character, using that $K$ is the centre of~$G$ due to the implication (\ref{prop:rep:E}') $\Rightarrow$ (\ref{prop:rep:C}') and the discussion preceding this proposition. The only non-trivial checks are the following ones where $k,k' \in K$: $\chi'(k \cdot \sigma' k') = \chi'(\sigma' k k') = s \cdot \chi(k k') = s \cdot \chi(k) \cdot \chi(k') = \chi'(\sigma' k) \cdot \chi'(k')$ and
$\chi'(\sigma' k \cdot \sigma' k') = \chi'(\sigma'^2 k k') = \chi(\sigma'^2) \cdot \chi(k) \cdot \chi(k') = s^2 \cdot \chi(k) \cdot \chi(k') = \chi'(\sigma' k) \cdot \chi'(\sigma'k')$.

Let $\tau \in H \setminus H'$. Then $\chi'^\tau(k) = \chi'(\tau^{-1} k \tau) = \chi(\tau^{-1} k \tau) = \chi(\tau^{-1})\cdot  \chi(k)\cdot  \chi(\tau) = \chi(k)$ for all $k \in K$.
Furthermore, we have
$\chi'^\tau(\sigma')^{-1} \cdot \chi'(\sigma') = \chi'(\tau^{-1}\sigma'^{-1}\tau\sigma') = \chi(\tau^{-1}\cdot\sigma'^{-1}\tau\sigma') = \chi(\tau^{-1}) \cdot \chi^{\sigma'}(\tau) = \epsilon_\chi(\tau) = -1$, as $\tau^{-1}\sigma'^{-1}\tau\sigma' \in K$ since $G/K$ is commutative.

Let $\rho' = \Ind_{H'}^G(\chi')$. We check $\tr(\rho(g))=\tr(\rho'(g))$ for all $g \in G$ by distinguishing cases.
If $g=k \in K$, then the previous computations imply $\tr(\rho'(k)) = \chi'(k) + \chi'^\tau(k) = \chi(k) + \chi^\sigma(k) = \tr(\rho(k))$ for $k \in K$. If $g \in H \setminus H'$, then $\tr(\rho'(g))=0$ and also $\tr(\rho(g))=\chi(g)+\chi^\sigma(g)=0$.
Similarly, if $g \in H' \setminus H$, then $\tr(\rho(g))=0$ and also $\tr(\rho'(g))=\chi'(g)+\chi'^\tau(g)=0$.
Finally, if $g \in G \setminus(H \cup H')$, then $\tr(\rho(g))=\tr(\rho'(g))=0$.
Consequently, $\rho$ and $\rho'$ are irreducible with the same character, whence they are isomorphic.

\underline{(\ref{prop:rep:F}') $\Rightarrow$ (\ref{prop:rep:A}'):} Trivial.

\underline{(\ref{prop:rep:A}') $\Rightarrow$ (\ref{prop:rep:C}'):}
Let $\tau \in H \setminus H'$ and $\sigma' \in H' \setminus H$.
Then we find $\tr(\rho(\tau))=0=\chi(\tau) + \chi^\sigma(\tau)$, whence $\epsilon_\chi(\tau) = -1$.
Restricting~$\rho$ to~$K$, we find $\chi|_K \oplus \chi^{\sigma'}|_K \cong \chi'|_K \oplus \chi'^\tau|_K$ and so by Maschke's theorem $\chi|_K \cong \chi'|_K$ or $\chi|_K \cong \chi'^\tau|_K$.
In the former case, the computation $\chi^{\sigma'}(k) = \chi(\sigma'^{-1} k \sigma') = \chi'(\sigma'^{-1} k \sigma') = \chi'(k) = \chi(k)$ for $k \in K$ shows $\chi^{\sigma'}|_K = \chi|_K$.
In the latter case, we adapt it to $\chi^{\sigma'}(k) = \chi(\sigma'^{-1} k \sigma') = \chi'^\tau(\sigma'^{-1} k \sigma') = \chi'^\tau(k) = \chi(k)$ for $k \in K$ and reach the same conclusion. Consequently, $\epsilon_\chi(k)=1$ for all $k \in K$, showing that $\epsilon_\chi$ is quadratic.
\end{proof}

Suppose that $\rho$ satisfies the equivalent conditions of Proposition~\ref{prop:rep} and denote by $\chi_K$ the restriction to~$K$ of $\chi$, which is independent of the choice of index~$2$ subgroup of~$G$. Then we have
\begin{equation}\label{eq:trdet}
\tr(\rho(g))= \begin{cases}
2 \chi_K(g) & \textnormal{if }g \in K, \\
0           & \textnormal{if }g \notin K,
\end{cases} \textnormal{ and }
\det(\rho(g))= \begin{cases}
 \chi_K(g^2) & \textnormal{if }g \in K, \\
-\chi_K(g^2) & \textnormal{if }g \notin K.
\end{cases}
\end{equation}

We include here an explicit description of the possible representations $\rho$, satisfying the equivalent statements of Proposition~\ref{prop:rep} and determine their images. We let $H=H_1, H_2, H_3$ be the three subgroups of~$G$ of index~$2$ and $\chi_i:H_i \to \CC^\times$ the characters such that $\rho \cong \Ind_{H_i}^G(\chi_i)$ for $i=1,2,3$. Further, we choose $\sigma_i \in H_i \setminus K$ for $i=1,2,3$ with $\sigma_3 = \sigma_1\sigma_2$. Then we have the coset decomposition $G = K \sqcup \sigma_1 K \sqcup \sigma_2 K \sqcup \sigma_1\sigma_2 K$, and the consequence that $G$ is generated by $K$ and any two of $\sigma_1,\sigma_2, \sigma_1\sigma_2=\sigma_3$. For the choice $H=H_1$ and $\sigma=\sigma_2$, we obtain from \eqref{eq:induced} for any $k \in K$:
\begin{align*}
\rho(k) &= \mat {\chi_K(k)}00{\chi_K(k)}, & \rho(\sigma_1 k) &= \mat {\chi_1(\sigma_1)\chi_K(k)}00{-\chi_1(\sigma_1)\chi_K(k)},\\
\rho(\sigma_2 k) &= \mat 0{\chi_K(\sigma_2^2)\chi_K(k)}{\chi_K(k)}0, & \rho(\sigma_1 \sigma_2 k) &= \mat 0{\chi_1(\sigma_1)\chi_K(\sigma_2^2)\chi_K(k)}{-\chi_1(\sigma_1)\chi_K(k)}0.
\end{align*}
We immediately see $\ker(\rho) = \ker(\chi_K)$. For the sake of easing notation, we replace $G$ by $G/\ker(\rho) \cong \Image(\rho)$ and $K$ by $K/\ker(\chi_K) \cong \Image(\chi_K)$, making $\rho$ and $\chi_K$ faithful representations. From the matrix description, we further deduce that $K$ lies in the centre of~$G$ and hence is cyclic $K \cong C_n$ of order~$n=2^r m$ with $m$ odd and $r \ge 1$ because $\chi_K$ takes the value~$-1$. Consequently, the image of the projectivisation of~$\rho$ (i.e.\ its image in $\PGL_2(\CC)$) is isomorphic to $G/K \cong C_2 \times C_2$. Moreover, we see that the $H_i$ are abelian of order $2n$, whence they are either cyclic $C_{2n}$ or isomorphic to $C_2 \times K$.
By possibly replacing $\sigma_i$ by an appropriate substitute in $\sigma_i K$, we assume that $\sigma_i$ has order~$2^{r+1}$ in the former case, and is an involution in the latter case. Write $\zeta_q = \exp(2 \pi i /q)$ for $q \in \ZZ$.
Then we get the following explicit set of generators for $H_1$ and $H_2$.

\noindent
\begin{tabular}{|c|c|c|}
\hline
Group & $C_{2n} \cong C_{2^{r+1}} \times C_m$ & $C_2 \times C_n \cong C_2 \times C_{2^r} \times C_m$ \\
\hline
\hline
$H_1$  & $c_1 = \rho(\sigma_1) = \mat {\zeta_{2^{r+1}}}00{-\zeta_{2^{r+1}}}, k' = \mat {\zeta_m}00{\zeta_m}$ & $\iota_1 = \rho(\sigma_1) = \mat 100{-1}, k= \mat {\zeta_{2^r}}00{\zeta_{2^r}}, k'$   \\
\hline
$H_2$  & $c_2 = \rho(\sigma_2) = \mat 0{\zeta_{2^r}}10, k'= \mat {\zeta_m}00{\zeta_m}$                & $\iota_2 = \rho(\sigma_2) = \mat 0110, k= \mat {\zeta_{2^r}}00{\zeta_{2^r}}, k'$ \\
\hline
\end{tabular}\\

In order to ease notation further, we write $K' = \langle k' \rangle \cong C_m$ and notice $H_i = \tilde{H}_i \times K'$ for $i=1,2,3$ and groups $\tilde{H}_i$. Consequently, we have $G = \tilde{G} \times K'$ for some group $\tilde{G}$. In order to compute $G$ and $H_3$, it thus suffices to determine $\tilde{G}$ and $\tilde{H}_3$ for given $\tilde{H}_i$ for $i=1,2$, which we do in the following proposition, which is proved by simple matrix calculations. For the naming of the groups, we follow Tim Dokchitser's page on GroupNames (\cite{GroupNames}).

\begin{prop}\label{prop:image}
We suppose that $\rho$ satisfies the equivalent conditions of Proposition~\ref{prop:rep}. Then the image of $\rho$ is isomorphic to $\tilde{G} \times C_m$ for some odd integer $m \ge 1$, where $\tilde{G}$ is a $2$-group of order $2^{r+2}$ with $r \ge 1$ containing three distinct index~$2$ subgroups $\tilde{H}_i$ with $\tilde{H}_i$ being cyclically generated by $c_i$ of order $2^{r+1}$ or generated by an involution $\iota_i$ together with a central element $k$ of order $2^r$, for $i=1,2,3$.

More precisely, following the notation introduced above, in the four cases corresponding to the choices of generators for $\tilde{H}_1$ and $\tilde{H}_2$, the groups $\tilde{H}_3$ and $\tilde{G}$ are explicitly given as follows.\\
\noindent
\begin{tabular}{|c||c|c|c||c|c|}
\hline
 & $r$ & $\tilde{H}_1$ & $\tilde{H}_2$ & Generator(s) of $\tilde{H}_3$ & Presentation of $\tilde{G}$ \\
\hline
\hline
$1$ & $ 1$    & $\iota_1,k$ & $\iota_2, k$ & $c_3 = \iota_1 \iota_2 = \mat 01{-1}0$ & $\langle \iota_1, c_3 \;|\; 1=\iota_1^2 = c_3^4 = (\iota_1 c_3)^2 \rangle$ \\
\hline
$2$ & $\ge 2$ & $\iota_1,k$ & $\iota_2, k$ & $\iota_3 = \iota_1 \iota_2 k^{2^{r-2}}= \mat 0{\zeta_4}{-\zeta_4}0, k$ & \begin{minipage}{6cm}\vspace*{0.1cm} $\langle \iota_1, \iota_3, k \;|\; 1=\iota_1^2 = \iota_3^2 = k^{2^r} = [\iota_1,\iota_3] k^{2^{r-1}} = [\iota_1,k] = [\iota_3,k]\rangle$ \end{minipage}  \\
\hline
$3$ & $ 1$    & $c_1$ & $c_2$ & $\iota_3 = c_1 c_2 = \mat 0{\zeta_4}{-\zeta_4}0, k$ & $\langle c_1,c_2 \;|\; 1=c_1^4=c_1^2c_2^{-2}= c_2 c_1 c_2^{-1}c_1\rangle$ \\
\hline
$4$ & $\ge 2$ & $c_1$ & $c_2$ & \begin{minipage}{3.6cm}\vspace*{0.1cm} $\iota_3=c_1 c_2 k^{2^{r-2}-1}= \zeta_4 \mat 0{\zeta_{2^{r+1}}}{-\zeta_{2^{r+1}}^{-1}}0, k$ \end{minipage} & $\langle c_2, \iota_3 \;|\; 1=c_2^{2^{r+1}}=\iota_3^2=\iota_3 c_2 \iota_3 c_2^{2^r-1}  \rangle$ \\
\hline
$5$ & $ \ge 1$ & $c_1$ & $\iota_2$  & $c_3 = c_1 \iota_2 = \mat 0{\zeta_{2^{r+1}}}{-\zeta_{2^{r+1}}}0 $ & $\langle c_1, \iota_2 \;|\; 1=c_1^{2^{r+1}}=\iota_2^2=\iota_2 c_1 \iota_2 c_1^{2^r-1}\rangle$ \\
\hline
$6$ & $ \ge 1$ & $\iota_1$ & $c_2$   & $c_3 = \iota_1 c_2 = \mat 0{\zeta_{2^r}}{-1}0$ & $\langle \iota_1, c_2 \;|\; 1=c_2^{2^{r+1}}=\iota_1^2=\iota_1 c_2 \iota_1 c_2^{2^r-1}\rangle$ \\
\hline
\end{tabular}\\
In line~$1$, the group $\tilde{G}$ is isomorphic to the dihedral group~$D_4$, in line~$2$ to the central product $D_4 \circ C_{2^r}$, in line~$3$ to the quaternion group~$Q_8$, and in lines $4-6$ to the modular maximal-cyclic group $M_{r+2}(2)$.
\end{prop}

\begin{rem}
\begin{enumerate}[(a)]
\item The group $M_3(2)$ is isomorphic to~$D_4$, whence $D_4$ also occurs in lines $5$ and $6$ for $r=1$.
\item Proposition~\ref{prop:image} can be viewed as a more explicit version of \cite[Proposition~18]{R-AlmAb}, where the naming of the groups differs: $D_4$ is written as $D_8$, the modular maximal-cyclic group $M_{r+2}(2)$ is denoted $N_{2^{r+2}}$, and the central product $D_4 \circ C_{2^r}$ is denoted $DT_{2^{r+2}}$.
\item In all cases there is an involution of determinant~$-1$ except in line~$3$, {\it i.e.} if the image of the representation is isomorphic to $Q_8$. This is indeed the only case when $\det(\rho)$ is the trivial character because non-scalar involutions have determinant~$1$. A similar result has been obtained by Rohrlich (\cite[Proposition~3]{R-Quat}).
\end{enumerate}
\end{rem}

\section{Theta series of weight one and their Galois representations}

In this section, we collect statements about Hecke and Dirichlet characters, their associated theta series and their Galois counterparts, which we will need to pass from the representation theoretic result to a proof of Köhler's conjecture. In particular, we will only treat Dirichlet characters and do not develop the larger theory of Hecke characters.

Throughout this section, let $K$ be a quadratic extension of~$\QQ$.

\subsection{Dirichlet characters}

We base this subsection on \cite[VII, \S 6, especially 6.1, 6.2, 6.8, 6.9]{N-AZT}.

\begin{defi}
Let $\fm$ be an integral ideal of $\cO_K$. We denote by $J(\fm)$ the group of all fractional ideals coprime to $\fm$ and by $P(\fm)$ its subgroup consisting of all principal fractional ideals $(a)$ such that $a \equiv 1 \pmod{\fm}$\footnote{For $a=\frac{b}{c}$ with $b,c \in \cO_K$ coprime to~$\fm$, the congruence $a \equiv 1 \pmod{\fm}$ means $b \equiv c \pmod{\fm}$.} and $a \in K$ is totally positive ({\it i.e.} $\tau(a) > 0$ for all real embeddings $\tau: K \to \RR$).

A \textbf{{\em Dirichlet character of modulus~$\fm$}} is a group homomorphism $\xi \colon J(\fm) \rightarrow \CC^\times$ such that $\xi|_{P(\fm)} = 1$. By convention, for a fractional ideal $\fa$ we set $\xi(\fa)=0$ if $\fa$ and $\fm$ are not coprime.
\end{defi}

Every such Dirichlet character $\xi$ uniquely determines a character $\xi_f: (\cO_K/\fm)^\times \to \CC^\times$ and $p_\tau \in \{0,1\}$ for every field embedding $\tau: K \hookrightarrow \CC$ with $p_\tau = 0$ if $\tau$ is complex, by restricting to principal ideals via the equality
\begin{equation}\label{eq:ptau}
\xi((a)) = \xi_f(a) \cdot \prod_{\tau} \big(\frac{\tau(a)}{|\tau(a)|}\big)^{p_\tau}
\end{equation}
for any $0 \neq a \in \cO_K$, where $\tau$ runs over the embeddings $K \hookrightarrow \CC$.

If $\fn$ is an ideal of $\cO_K$ dividing $\fm$, then by \cite[VII.6.2]{N-AZT} $\xi_f$ factors through $(\cO_K/\fn)^\times$ if and only if $\xi$ can be extended to a Dirichlet character $\tilde{\xi}$ of modulus~$\fn$. The extension is unique because any ideal $\fa \in J(\fn)$ can be written as $\fa = \fb \cdot (a)$ for some $\fb \in J(\fm)$ and $(a) \in P(\fn)$, in which case we have $\tilde{\xi}(\fa) = \xi(\fb) \cdot \tilde{\xi}((a)) = \xi(\fb) \cdot \tilde{\xi}_f(a) \cdot \prod_{\tau} \big(\frac{\tau(a)}{|\tau(a)|}\big)^{p_\tau}$. The $p_\tau$ remain unchanged in the extension, which follows from~\eqref{eq:ptau}.

\begin{defi}
The \textbf{\em conductor} of a Dirichlet character $\xi$ of modulus $\fm$ is defined to be the largest integral ideal $\ff$ of $\cO_K$ dividing $\fm$ such that $\xi_f$ factors through $(\cO_K/\ff)^\times$.
A Dirichlet character $\xi$ of modulus~$\fm$ is called \textbf{{\em primitive}} if $\fm$ is itself the conductor of $\xi$.
\end{defi}

Consequently, any Dirichlet character $\xi$ of modulus $\fm$ extends uniquely to a primitive Dirichlet character $\tilde{\xi}$ for $\ff \mid \fm$.

\subsection{Theta series}

Here, following \cite[\S 5.2]{Koehler-book}, we present the definition of Hecke theta series associated with Dirichlet characters $\xi$ of modulus~$\fm$ for a quadratic extension $K$ of~$\QQ$.
Denoting the norm of an ideal $\fa \subset \cO_K$ by $N(\fa)$, one defines the Hecke theta series corresponding to $\xi$ as
\begin{equation}\label{eq:an}
\Theta_{1}(K, \xi, z)
= \sum_{\fa \subseteq \cO_K}\xi(\fa)e^{2\pi i N(\fa) z} = \sum_{n=1}^{\infty} a_n e^{2\pi i nz},
\textnormal{ where }
a_n = \sum_{\substack{\fa \subseteq \cO_K\\ N(\fa) = n}} \xi(\fa).
\end{equation}

From \eqref{eq:an} and the multiplicativity of the norm and the Dirichlet character~$\xi$, one obtains the following explicit computations of the Fourier coefficients of these Hecke theta series.

\begin{lem}\label{lemma:7}
\begin{enumerate}[(a)]
    \item For primes $p$ in $\ZZ$ inert in~$K$:
    \[
    a_{p^r} = \begin{cases} 0                     & \textnormal{ if } 2 \nmid r, \\
                            \xi(p\cO_K)^{r/2} & \textnormal{ if } 2 \mid r,
              \end{cases}
    \]
    where $r \in \ZZ_{\ge 1}$.

    \item For primes $p$ in $\ZZ$ ramified in~$K$:
    \[
    a_{p^r} = \xi(\fp)^r,
    \]
    where $\fp$ is the unique prime ideal above $p$ in $K$.

    \item For primes $p$ in $\ZZ$ split in~$K$:
    \begin{align*}
    a_p &= \xi(\fp_1) + \xi(\fp_2) & \\
    a_{p^r} &= \xi(\fp_1^r) + \xi(\fp_1^{r-1}\fp_2) + \cdots + \xi(\fp_1\fp_2^{r-1}) +\xi(\fp_2)^r & \textnormal{ for } r \ge 2,
    \end{align*}
    where $\fp_1$ and $\fp_2$ are the distinct prime ideals lying above~$p$, {\em i.e.} $p\cO_K = \fp_1 \fp_2$.
    \item $a_{nm}=a_na_m$ for all $n,m \in \ZZ_{\ge 1}$ such that $\gcd(n,m)=1$.
\end{enumerate}
\end{lem}

\begin{thm}\label{thm 2}\cite[Theorems 5.1, 5.3]{Koehler-book}
Let $K$ be a quadratic field with discriminant $D$ and $\xi$ a Dirichlet character of modulus $\fm$ for~$K$ given by $\xi_f$ and $p_\tau \in \{0,1\}$ for all $\tau: K \hookrightarrow \CC$ with $p_\tau = 0$ for complex~$\tau$.

If $K$ is imaginary, assume that $\xi$ is not induced from a Dirichlet character of~$\QQ$ through the norm.

If $K$ is real, assume that exactly one of the two $p_\tau$ equals~$1$.

Then $\Theta_1(K,\xi, z)$ is a cuspidal modular form of weight $1$ for the congruence group $\Gamma_1(|D|N(\fm))$ and a normalised Hecke eigenform, {\em i.e.} a common eigenfunction of all Hecke operators $T_n$ for $n \in \ZZ_{\ge 1}$ with $a_1 = 1$.
Moreover, if $\xi$ is primitive, then $\Theta_1(K,\xi, z)$ is a newform.
\end{thm}

\subsection{Galois representations attached to theta series of weight one}\label{sec:galrep}
In this section, we collect the statements from class field theory, which we need to pass between Dirichlet and Galois characters, and construct the Galois representation attached to the theta series from Theorem~\ref{thm 2}.
We continue to work with a quadratic extension $K$ of~$\QQ$. By $\Frob_p$ we denote any choice of Frobenius element at~$p$ in the Galois group of any Galois extension of~$\QQ$ and by $\Frob_\fp$ a Frobenius element at a prime $\fp \subset \cO_K$ in the Galois group of any Galois extension of~$K$.

\begin{thm}[Artin Reciprocity Law] (\cite[III.9.2]{N-KKT}) For an integral ideal $\fm$ of $K$, there exists an abelian extension $K^\fm/K$, called the \textbf{\textit{narrow ray class field of modulus $\fm$} over $K$}, such that there is an isomorphism
\[
\left( \frac{ \left. K^\fm \middle| K \right. }{ \cdot } \right) \colon J(\fm)/P(\fm)\xrightarrow{\sim} \Gal(K^\fm/K); \fa \mapsto \left( \frac{ \left. K^\fm \middle| K \right. }{ \fa} \right)
\]
where $\left( \frac{ \left. K^\fm \middle| K \right. }{\fp} \right)=\Frob_\fp$ for all prime ideals $\fp$ of $K$ coprime to $\fm$.
Here $\left( \frac{ \left. K^\fm \middle| K \right. }{ \fa} \right)$ is the Artin symbol.
\end{thm}

The Artin reciprocity law thus allows us to obtain from any Dirichlet character $\xi$ of modulus $\fm$ for~$K$ a character $\chi$ of $\Gal(L/K)$ for any Galois extension $L/K$ containing $K^\fm$ via
\begin{equation}\label{eq:Frob}
\chi \colon \Gal(L/K) \twoheadrightarrow \Gal(K^\fm/K) \xrightarrow{\left( \frac{ \left. K^\fm \middle| K \right. }{ \cdot } \right)^{-1}} J(\fm)/P(\fm) \xrightarrow{\xi} \CC^\times \textnormal{ and }\chi(\Frob_\fp) = \xi(\fp)
\end{equation}
for any prime $\fp$ of~$K$.

Moreover, any character of the Galois group of a Galois extension $L$ of $K$ can be obtained from some Dirichlet character of $K$ in this way because any abelian extension of $K$ is contained in some $K^\fm$ by a generalisation of the theorem of Kronecker--Weber (\cite[III.7.9]{N-KKT}).

If $K$ is totally real and $L$ totally imaginary, then $\Gal(L/K)$ contains a complex conjugation $c_\tau$ for each embedding $\tau: K \hookrightarrow \RR$ (which is uniquely defined up to conjugation in $\Gal(L/K)$ and consequently unique in any abelian quotient of $\Gal(L/K)$). Via the inverse Artin reciprocity map, $c_\tau$ corresponds to an ideal class, which we denote $\fc_\tau$.
We have $\xi(\fc_\tau) = \chi(c_\tau) = (-1)^{p_\tau}$. This can, for instance, be seen by viewing the Dirichlet character as a character of the idèle class group (\cite[VII.6.14]{N-AZT}) and using the definition of the Artin reciprocity map for infinite places (\cite[\S II.5]{N-KKT}).

The following well-known theorem establishes the existence of a Galois representation attached to a theta series of weight one.
\begin{thm}\label{thm:main}
Let $\Theta_{1}(K, \xi, z)$ be a theta series of weight one coming from Theorem~\ref{thm 2}.
Let $L/\QQ$ be any Galois extension containing $K^\fm$ and let $\chi$ be the corresponding character of $\Gal(L/K)$ obtained from $\xi$ via the Artin reciprocity map.

Then the Galois representation
\[
\rho = \Ind_{\Gal(L/K)}^{\Gal(L/\QQ)}(\chi)  \colon \Gal(L/\QQ) \to \GL_2(\CC)
\]
is attached to $\Theta_{1}(K, \xi, z)$ in the sense that $\tr(\rho(\Frob_p)) = a_p$ for all primes~$p$ coprime to~$|D|N(\fm)$.
Moreover, the representation $\rho$ is odd, {\it i.e.} the determinant of complex conjugation equals~$-1$.
\end{thm}

\begin{proof}
The assumptions on $\xi$ imply that $\chi$ is different from its conjugate, when conjugated by any $\sigma \in \Gal(L/\QQ) \setminus \Gal(L/K)$. Consequently, by Mackey's criterion, $\rho$ is irreducible.
By Chebotarev's Density Theorem, $\rho$ is uniquely determined by the equality $\tr(\rho(\Frob_p)) = a_p$ for all primes~$p$ coprime to~$|D|N(\fm)$. We notice that such $p$ are either split or inert in~$K$, and handle these two cases one by one.

Suppose that $p$ is inert in~$K$. Then by Lemma~\ref{lemma:7}, we find $a_p=0$. Under the same assumption, $\Frob_p \in \Gal(L/\QQ) \setminus \Gal(L/K)$, whence \eqref{eq:induced} yields $\tr(\rho(\Frob_p)) = 0$.

Now consider that $p$ splits in~$K$, more precisely, $p \cO_K = \fp_1 \fp_2$. Then by Lemma~\ref{lemma:7}, we find $a_p=\xi(\fp_1) + \xi(\fp_2)$. On the Galois side, we have $\Frob_p \in \Gal(L/K)$ and we can assume without loss of generality $\Frob_p = \Frob_{\fp_1}$, whence $\sigma^{-1}\circ \Frob_p \circ \sigma = \Frob_{\fp_2}$. We thus find from \eqref{eq:induced} and \eqref{eq:Frob}
\[ \tr(\rho(\Frob_p)) = \chi(\Frob_{\fp_1}) + \chi(\Frob_{\fp_2}) = \xi(\fp_1) + \xi(\fp_2).\]

If $K$ is imaginary, then complex conjugation does not fix it; consequently, the conjugacy class of complex conjugation $c \in \Gal(L/\QQ)$ does not lie in the subgroup $\Gal(L/K)$, whence it is an involution of trace~$0$, implying that its determinant is~$-1$.
If $K$ is real, then for each embedding $\tau: K \to \RR$, we have a complex conjugation $c_\tau \in \Gal(L/K)$.
Moreover, $\sigma^{-1} \circ c_\tau \circ \sigma$ is a complex conjugation $c_{\tau \sigma}$ for the other embedding $\tau \sigma$, whence without loss of generality $\chi(c) = \chi(c_\tau) = \xi(\fc_\tau) = (-1)^{p_\tau}$ and $\chi^\sigma(c_\tau) = \chi(c_{\tau \sigma}) = (-1)^{p_{\tau \sigma}}$. This finally yields $\det(\rho(c)) = (-1)^{p_\tau + p_{\tau \sigma}} = -1$ by the condition on the $p_\tau$.
\end{proof}

\section{Köhler's conjecture}

In this section, we prove a slightly modified version of Köhler's conjecture (\cite[\S4 Conjecture]{Koehler-article}).

\begin{thm}\label{thm:conj}
Let $\Theta_1(K, \xi, z)$ be a Hecke theta series of weight $1$ on a quadratic field $K$ and with a Dirichlet character $\xi$ for~$K$.
Consider the following statements.
\begin{enumerate}[(A)]
    \item\label{thm:conj:A} There are a second quadratic field $K' \neq K$ and a Dirichlet character $\xi'$ for $K'$ such that we have the identity
    \[
        \Theta_1(K, \xi, z) \;=\; \Theta_1(K', \xi', z).
    \]
    
    \item\label{thm:conj:B} The character $\xi$ has the following properties:
    \begin{itemize}
        \item[(2)] The character $\xi$ is not induced from a Dirichlet character of~$\QQ$ through the norm. So, $\Theta_1(K, \xi, z)$ is a cusp form.
        \item[(3)] Let $\fp$ be a prime ideal in $\cO_{K}$, $\sigma(\fp)$ its conjugate and suppose that $\fp$ and $\sigma(\fp)$ are both relatively prime to the modulus of~$\xi$. Then we have
              \[
                  \xi(\fp) = \xi(\sigma(\fp)) \quad \textrm{or} \quad \xi(\fp) = -\xi(\sigma(\fp)).
              \]
    \end{itemize}

    \item\label{thm:conj:C} There are two quadratic fields $K' \neq K''$ different from~$K$ and Dirichlet characters $\xi'$ for $K'$ and $\xi''$ for $K''$
    such that we have
\[
   \Theta_1(K, \xi, z) \;=\; \Theta_1(K', \xi', z) \;=\; \Theta_1(K'', \xi'', z).
\]
\end{enumerate}

The statements \eqref{thm:conj:A} and \eqref{thm:conj:C} are equivalent.
Furthermore, statement \eqref{thm:conj:A} implies \eqref{thm:conj:B}.
If $\xi$ is assumed to be primitive, then all three statements are equivalent.
\end{thm}

\begin{rem}
\begin{enumerate}[(a)]
\item We chose to follow Köhler's original formulation quite closely in Theorem~\ref{thm:conj} but made some notational changes (e.g.\ Dirichlet characters vs.\ Hecke characters) and suppressed the explicit mentions of the moduli of the characters and the discriminants of the quadratic fields.

\item In Köhler's original formulation, there was an additional item (1) in \eqref{thm:conj:B}, which we removed as it need not always hold because our computation of the images of Galois representations shows that the characters can have arbitrarily large order.

\item If $\xi$ is not primitive, the implication \eqref{thm:conj:B} $\Rightarrow$ \eqref{thm:conj:A} need not hold.

To construct a counter example, let $\tilde{\xi}$ be a primitive Dirichlet character satisfying both conditions of~\eqref{thm:conj:B}. By the proved equivalence in the primitive case, we obtain quadratic fields $K'$, $K''$ and primitive Dirichlet characters $\tilde{\xi}'$, $\tilde{\xi}''$ for which \eqref{thm:conj:A} holds. We take $p$ to be a prime coprime to the conductor that is split in~$K$, but inert in~$K'$. The latter condition implies that the $a_p$-coefficient of $\Theta_1(K',\tilde{\xi}',z)$ equals~$0$, and the same holds true for $\Theta_1(K',\xi',z)$ for any Dirichlet character $\xi'$ with primitive character $\tilde{\xi}'$. For a counter example, it thus suffices to choose a modulus $\fm$ for~$K$ containing exactly one of the two primes of~$K$ above~$p$ because then $a_p(\Theta_1(K,\xi,z)) = \xi(\sigma(\fp)) \neq 0$.

\item The previous item might suggest that one should ask condition (3) in \eqref{thm:conj:B} for all primes~$\fp$, and not only those coprime to the modulus of~$\xi$. However, this condition is not implied by \eqref{thm:conj:A}.

The reason is the following. Let $\tilde{\xi}$ be a primitive Dirichlet character for~$K$ of conductor~$\ff$ satisfying all equivalent conditions of Theorem~\ref{thm:conj} and let $p$ be a prime which is coprime to~$\ff$, split in $K$ and inert in~$K'$.
If we now let $\xi$ be the non-primitive Dirichlet character obtained from $\tilde{\xi}$ by defining it on the modulus $\fm = \ff \fp$, where $\fp$ is one of the primes of~$K$ above~$p$, then $a_p(\Theta_1(K,\xi,z))=\xi(\sigma(\fp)) \neq 0$, whereas for every choice of Dirichlet character $\xi'$ on~$K'$ we have $a_p(\Theta_1(K',\xi',z))=0$.

\item In the case of $\xi$ (and, consequently, $\xi'$, $\xi''$) being primitive, the levels of the three Hecke theta series $\Theta_1(K, \xi, z) \;=\; \Theta_1(K', \xi', z) \;=\; \Theta_1(K'', \xi'', z)$ are equal to the conductor of $\rho\cong\rho{'}\cong\rho^{''}$, which by \cite[\S3]{taguchi} are equal to $N(\mathfrak{m}) \cdot D = N(\mathfrak{m}') \cdot D^{'} = N(\mathfrak{m}'') \cdot D^{''}$, where $D$, $D^{'}$ and $D^{''}$ are the discriminants of $K$, $K'$ and $K''$, and $\fm$, $\fm'$, $\fm''$ are the conductors of $\xi$, $\xi'$, $\xi''$, respectively.

The equalities of Hecke theta series for non-primitive Dirichlet characters are meant as equalities of functions; we do not specify the levels of the theta series individually.

\item In the case of \eqref{thm:conj:C}, exactly one of the three quadratic fields $K$, $K'$, $K''$ is real.

Indeed, if two among these fields were real, their composite would also be real. Consequently, any complex conjugation would fix this composite and by \eqref{eq:trdet} it would have determinant~$1$. This would imply that the Galois representation $\rho$ from Theorem~\ref{thm:main} is even, a contradiction.

There are cases when such a Galois representation $\rho$ is induced from three distinct real quadratic fields (for instance with image $D_4$ by number field 4.4.725.1 or with image $Q_8$ by number field 8.8.12230590464.1 in \cite{lmfdb}). However, it does not come from any Hecke theta series of weight~$1$.

\end{enumerate}
\end{rem}

\begin{proof}[Proof of Theorem~\ref{thm:conj}.]
The key idea is to reduce to the analogous statements about Galois representations in Proposition~\ref{prop:rep}.
To carry out the transition between Galois representations and theta series, we adopt the following notation.
Let $L$ be a finite Galois extension of $\QQ$, which need be `big enough' in a sense that is specified in each case.
Relative to this choice of~$L$, we set $G=\Gal(L/\QQ)$ and $H=\Gal(L/K)$ and, if applicable, $H' = \Gal(L/K')$ and $H'' = \Gal(L/K'')$.
If we deal with a Dirichlet character $\xi$ of modulus~$\fm$ for~$K$, we assume that $L$ contains $K^\fm$, the narrow ray class field of~$K$ of modulus~$\fm$ and we associate to it the Galois character $\chi: H \to \CC^\times$ as explained in \S\ref{sec:galrep}.
We proceed similarly for $\xi'$ and $\xi''$, if applicable.

\underline{(\ref{thm:conj:B}) $\Leftrightarrow$ (\ref{prop:rep:B}'):} Condition (2) translates to $\chi^\sigma \neq \chi$ and (3) is equivalent to $\chi^\sigma(h) = \pm \chi(h)$ for $h \in H$ as every $h \in H$ is equal to some $\Frob_\fp$ by Chebotarev's density theorem.

\underline{(\ref{thm:conj:A}) $\Rightarrow$ (\ref{prop:rep:A}'):} In this case, we assume that $L$ also contains the narrow ray class field of $K'$ of modulus~$\fm'$, where $\fm'$ is the modulus of~$\xi'$, so that we obtain the Galois character $\chi':H' \to \CC^\times$. The two Galois representations $\rho = \Ind_H^G(\chi)$ and $\rho' = \Ind_{H'}^G(\chi')$ have the same traces at $\Frob_p$ for all primes~$p$ not dividing the levels of $\Theta_1(K,\xi,z)$ and $\Theta_1(K',\xi',z)$ because these traces are the $a_p$-coefficients of the two theta series, which are assumed to be the same. Consequently $\rho$ and $\rho'$ are isomorphic, establishing (\ref{prop:rep:A}').
Using Proposition~\ref{prop:rep} and the triviality of (\ref{thm:conj:C}) $\Rightarrow$ (\ref{thm:conj:A}), we thus have the series of implications
\[ \textnormal{(\ref{thm:conj:C})} \Rightarrow \textnormal{(\ref{thm:conj:A})} \Rightarrow \textnormal{(\ref{prop:rep:A}')} \Leftrightarrow \textnormal{(\ref{prop:rep:C}')} \Leftrightarrow  \textnormal{(\ref{prop:rep:B}')} \Leftrightarrow \textnormal{(\ref{thm:conj:B}).}\]

\underline{(\ref{prop:rep:A}') $\Rightarrow$ (\ref{thm:conj:A}) for $\xi$ primitive:}
The subgroup $H' \subset G$ of index~$2$ corresponds to a quadratic field~$K' \neq K$ and the character $\chi':H' \to \CC^\times$ is thus the character of some abelian extension of~$K'$. We associate with it its unique primitive Dirichlet character~$\xi'$, from which we obtain the theta series $\Theta_1(K',\xi',z)$, which is a newform, exactly as $\Theta_1(K,\xi,z)$ due to the primitivity of the characters.
The equivalence of the Galois representations $\Ind_H^G(\chi)$ and $\Ind_{H'}^G(\chi')$ means that the $a_p$-coefficients of both theta series agree for all $p$ except possibly finitely many. However, those suffice to determine newforms uniquely (\cite[Theorem 13.3.9]{CS}), whence we conclude (\ref{thm:conj:A}).
Exactly the same arguments applied to $H''$ show (\ref{prop:rep:C}') $\Rightarrow$ (\ref{thm:conj:C}) for primitive~$\xi$, establishing the equivalence of \eqref{thm:conj:A}, \eqref{thm:conj:B} and \eqref{thm:conj:C} in this case.

\underline{(\ref{thm:conj:A}) $\Rightarrow$ (\ref{thm:conj:C}):}
Denote by $\tilde{\xi}$ and $\tilde{\xi}'$ the primitive characters of $\xi$ and $\xi'$, respectively, and let $\ff$ and $\ff'$ be their conductors.
By the established equivalence for primitive characters we thus have a primitive Dirichlet character $\tilde{\xi}''$ on the quadratic field~$K''$ such that the three theta series $\Theta_1(K,\tilde{\xi},z)$, $\Theta_1(K',\tilde{\xi}',z)$ and $\Theta_1(K'',\tilde{\xi}'',z)$ are equal. Let $\ff''$ be the conductor of $\tilde{\xi}''$. We will now define an ideal $\fa''$ of $K''$ such that the theta series $\Theta_1(K'',\xi'',z)$ equals $\Theta_1(K,\xi,z) = \Theta_1(K',\xi',z)$ for the Dirichlet character $\xi''$ of modulus~$\fm'' = \ff'' \fa''$ obtained from $\tilde{\xi}''$.

For a prime~$p$, we write $\tilde{a}_p =a_p(\Theta_1(K,\tilde{\xi},z))$, $a_p = a_p(\Theta_1(K,\xi,z))$ and $a_p'' = a_p(\Theta_1(K'',\xi'',z))$.
We proceed by case distinctions and base our arguments on a simple observation derived from~\eqref{eq:an}: If non-zero, then $a_p$ and $\tilde{a}_p$ are either a root of unity or a sum of two roots of unity. Moreover, in all cases $a_p$ can be obtained from $\tilde{a}_p$ by dropping a subset of the summands because $\xi(\fp)=0$ or $\xi(\fp)=\tilde{\xi}(\fp)$ for all primes $\fp$ of~$K$ above~$p$.

\noindent\underline{Assume $a_p = \tilde{a}_p$:} Then we take $\fa''$ coprime to~$p$, so that $\xi''$ and $\tilde{\xi}''$ have the same values at the prime(s) of $K''$ above~$p$, ensuring $a_p=a_p''$.

\noindent\underline{Assume $0 = a_p \neq \tilde{a}_p$:} In that case, we let $p\cO_{K''}$ divide~$\fa''$, making sure that $\xi''$ evaluates to~$0$ at the prime(s) of $K''$ above~$p$, whence $a_p''=0$.

\noindent\underline{Assume $0 \neq a_p \neq \tilde{a}_p$:} In this case, $a_p$ is equal to a root of unity, whereas $\tilde{a}_p$ is equal to a sum of two roots of unity. The latter implies that $p$ is split in~$K$ and coprime to the conductors of all three primitive Dirichlet characters $\tilde{\xi}, \tilde{\xi}',\tilde{\xi}''$.
Furthermore, $p$ cannot be inert in~$K'$, as this would imply $a_p=0$ via the theta series of~$\xi'$. Neither can $p$ be ramified in~$K'$, as that would imply that $\tilde{a}_p$ is equal to a root of unity via the theta series of~$\tilde{\xi}'$. From the splitting of $p$ in $K$ and $K'$, we conclude the splitting of~$p$ in $K''$.
Proposition~\ref{prop:rep} tells us that the Galois characters corresponding to $\tilde{\xi}, \tilde{\xi}',\tilde{\xi}''$ have the same restriction to the Galois group of $KK'$. Consequently, the three primitive characters have the same value $\zeta$ at both primes above~$p$ and so $\tilde{a}_p = 2\zeta$ and $a_p = \zeta$. It now suffices to choose one prime $\fp''$ of $K''$ above~$p$ and to let $\fp''$ divide~$\fa''$, but to ensure that the conjugated prime does not divide~$\fa''$. Then $a_p''=\zeta$.
\end{proof}

\bibliography{Bib-Koehler}
\bibliographystyle{alpha}

\bigskip

\noindent 
Mahima Kumar, Gabor Wiese\\
Department of Mathematics\\
University of Luxembourg \\
Luxembourg \\
\url{koomar.mahima@gmail.com},
\url{gabor.wiese@uni.lu}\\

\end{document}